\documentclass[11pt,reqno]{amsart}

\usepackage{amsaddr}
\usepackage{amssymb}
\usepackage{amsmath}
\usepackage{amsthm}

\usepackage{a4wide}
\usepackage{url}

\newcommand{\N}{\mathbb{N}}
\newcommand{\R}{\mathbb{R}}
\newcommand{\C}{\mathbb{C}}

\newcommand{\fonction}[5]{\begin{array}[t]{lrcl}#1 :&#2 &\longrightarrow &#3\\
&#4& \longmapsto &#5 \end{array}}
\newcommand{\fonctionsansdef}[3]{#1 : #2 \longrightarrow #3}

\newtheorem{theorem}{Theorem}
\newtheorem{proposition}[theorem]{Proposition}
\newtheorem{definition}[theorem]{Definition}
\newtheorem{remark}[theorem]{Remark}

\AtBeginDocument{{\noindent\small
This is a preprint of a paper whose final and definite form is published in\newline 
\emph{Discrete Dynamics in Nature and Society}, ISSN: 1026-0226 (Print),
ISSN: 1607-887X (Online),\newline 
available online at \url{http://dx.doi.org/10.1155/2016/8073023}.\newline
Submitted 23-Jan-2016; Revised 21-Feb-2016; Accepted 28-April-2016.\newline 
Cite this work as: Discrete Dyn. Nat. Soc. 2016 (2016), Art. ID 8073023, 8~pp. 
DOI: 10.1155/2016/8073023}
\vspace{9mm}}

\title[Helmholtz theorem for nondifferentiable Hamiltonian systems]{Helmholtz
theorem for nondifferentiable\\
Hamiltonian systems in the framework\\
of Cresson's quantum calculus}

\author[F. Pierret]{Fr\'ed\'eric Pierret}
\address{Institut de M\'ecanique C\'eleste et de Calcul des \'Eph\'em\'erides,\\
Observatoire de Paris, 75014 Paris, France\\
{\tt frederic.pierret@obspm.fr}}

\author[D. F. M. Torres]{Delfim F. M. Torres}
\address{Center for Research and Development in Mathematics and Applications (CIDMA),
Department of Mathematics, University of Aveiro, 3810-193 Aveiro, Portugal\\
{\tt delfim@ua.pt}}

\begin{document}

\begin{abstract}
We derive the Helmholtz theorem for nondifferentiable Hamiltonian systems
in the framework of Cresson's quantum calculus. Precisely, we give a theorem
characterizing nondifferentiable equations, admitting a Hamiltonian formulation.
Moreover, in the affirmative case, we give the associated Hamiltonian.
\end{abstract}



\maketitle

% ------------------------------------------

\section{Introduction}

Several types of quantum calculus are available in the literature, including
Jackson's quantum calculus \cite{MR1865777,MR2966852}, Hahn's quantum calculus
\cite{MR2861326,MR3028184,MR2733985}, the time-scale $q$-calculus
\cite{bohn,MR2793813}, the power quantum calculus \cite{MR2918250}, and the
symmetric quantum calculus \cite{MR3031158,MR3110294,MR3110279}. Cresson
introduced in 2005 his quantum calculus on a set of H\"{o}lder functions
\cite{MR2138974}. This calculus attracted attention due to its applications
in physics and the calculus of variations and has been further developed
by several different authors (see \cite{MR3338133,MR2546784,MR2738030,MR2549615}
and references therein). Cresson's calculus of 2005 \cite{MR2138974} presents,
however, some difficulties, and in 2011 Cresson and Greff improved it
\cite{cresson_greff,MR2798411}. Indeed, the quantum calculus of \cite{MR2138974} 
let a free parameter, which is present in all the computations. Such parameter
is certainly difficult to interpret. The new calculus of \cite{cresson_greff,MR2798411}
bypasses the problem by considering a quantity that is free of extra parameters 
and reduces to the classical derivative for differentiable functions.
It is this new version of 2011 that we consider here, with a brief review 
of it being given in Section~\ref{sec:reminder}. Along the text, 
by \emph{Cresson's calculus} we mean this quantum version of 2011
\cite{cresson_greff,MR2798411}. For the state of the art on the quantum calculus
of variations we refer the reader to the recent book \cite{MR3184533}.
With respect to Cresson's approach, the quantum calculus of variations 
is still in its infancy: see \cite{MR3338133,MR2653992,MR2861739,cresson_greff,MR2798411,MR3040924}.
In \cite{cresson_greff} nondifferentiable Euler--Lagrange equations are used
in the study of PDEs. Euler--Lagrange equations for variational functionals
with Lagrangians containing multiple quantum derivatives, depending on a parameter
or containing higher-order quantum derivatives, are studied in \cite{MR2653992}.
Variational problems with constraints, with one and more than one independent variable,
of first and higher-order type, are investigated in \cite{MR2861739}. Recently,
problems of the calculus of variations and optimal control with time delay
were considered \cite{MR3040924}. In \cite{MR2798411}, a Noether type theorem
is proved but only with the momentum term. This result is further extended in
\cite{MR3294632} by considering invariance transformations that also change
the time variable, thus obtaining not only the generalized momentum term
of \cite{MR2798411} but also a new energy term. In \cite{MR3338133},
nondifferentiable variational problems with a free terminal point,
with or without constraints, of first and higher-order, are investigated.
Here, we continue to develop Cresson's quantum calculus in obtaining
a result for Hamiltonian systems and by considering the so-called
inverse problem of the calculus of variations.

A classical problem in Analysis is the well-known {\it Helmholtz's inverse
problem of the calculus of variations}: find a necessary and sufficient
condition under which a (system of) differential equation(s) can be written
as an Euler--Lagrange or a Hamiltonian equation and, in the affirmative case,
find all possible Lagrangian or Hamiltonian formulations. This condition
is usually called the \textit{Helmholtz condition}.
The Lagrangian Helmholtz problem has been studied and solved by Douglas
\cite{doug}, Mayer \cite{maye} and Hirsch \cite{hirs,hirs2}. The Hamiltonian
Helmholtz problem has been studied and solved, up to our knowledge,
by Santilli in his book \cite{santilli}. Generalization of this problem
in the \emph{discrete calculus} of variations framework has been done in
\cite{bourdin-cresson,hydon1}, in the discrete Lagrangian case.
In the case of \emph{time-scale calculus}, that is, a mixing between continuous
and discrete subintervals of time, see \cite{MR3198223} for a
necessary condition for a dynamic integrodifferential equation
to be an Euler--Lagrange equation on time-scales. For the Hamiltonian case
it has been done for the discrete calculus of variations in \cite{opri1}
using the framework of \cite{mars} and in \cite{cresson-pierret2} using
a discrete embedding procedure derived in \cite{cresson-pierret1}. In the case
of \emph{time-scale calculus} it has been done in \cite{pierret_helmholtz_ts};
for the \emph{Stratonovich stochastic calculus} see \cite{pierret_helmholtz_sto}.
Here we give the Helmholtz theorem for Hamiltonian systems in the case of
\emph{nondifferentiable Hamiltonian systems} in the framework of Cresson's
quantum calculus. By definition, the nondifferentiable calculus extends
the differentiable calculus. Such as in the discrete, time-scale,
and stochastic cases, we recover the same conditions
of existence of a Hamiltonian structure.

The paper is organized as follows. In Section~\ref{sec:reminder},
we give some generalities and notions about the nondifferentiable
calculus introduced in \cite{cresson_greff}, the so-called
\emph{Cresson's quantum calculus}. In Section~\ref{sec:reminderhamsys},
we remind definitions and results about classical and nondifferentiable
Hamiltonian systems. In Section~\ref{sec:nondiffhelmholtz}, we give
a brief survey of the classical Helmholtz Hamiltonian problem and then
we prove the main result of this paper---the nondifferentiable Hamiltonian
Helmholtz theorem. Finally, we give two applications of our results in
Section~\ref{sec:Appl}, and we end in Section~\ref{sec:conclu}
with conclusions and future work.

% ------------------------------------------

\section{Cresson's Quantum Calculus}
\label{sec:reminder}

We briefly review the necessary concepts and results
of the quantum calculus \cite{cresson_greff}.

% ------------------------------------------

\subsection{Definitions}

Let $\mathbb{X}^d$ denote the set $\mathbb{R}^{d}$ or $\mathbb{C}^{d}$,
$d \in \mathbb{N}$, and $I$ be an open set in $\mathbb{R}$ with $[a,b]\subset I$,
$a<b$. We denote by $\mathcal{F}\left(I,\mathbb{X}^d\right)$ the set of functions
$f:I \rightarrow \mathbb{X}^d$ and by $\mathcal{C}^{0}\left(I,\mathbb{X}^d\right)$
the subset of functions of $\mathcal{F}\left(I,\mathbb{X}^d\right)$ 
which are continuous.

\begin{definition}[H\"{o}lderian functions \cite{cresson_greff}]
Let $f\in \mathcal{C}^0\left(I, \mathbb{R}^{d}\right)$. Let $t \in I$.
Function $f$ is said to be $\alpha$-H\"{o}lderian, $0<\alpha<1$, 
at point $t$ if there exist positive constants $\epsilon>0$ and $c > 0$
such that $|t-t'|\leqslant\epsilon$ implies 
$\|f(t)-f(t')\|\leqslant c |t-t'|^{\alpha}$ for all $t' \in I$, 
where $\|\cdot\|$ is a norm on $\mathbb{R}^{d}$.
\end{definition}

The set of H\"{o}lderian functions of H\"{o}lder exponent $\alpha$, for some
$\alpha$, is denoted by $H^\alpha(I,\mathbb{R}^{d})$. The quantum derivative
is defined as follows.

\begin{definition}[The $\epsilon$-left and $\epsilon$-right
quantum derivatives  \cite{cresson_greff}]
Let $f\in \mathcal{C}^{0}\left(I, \mathbb{R}^{d}\right)$. For all $\epsilon>0$,
the $\epsilon$-left and $\epsilon$-right quantum derivatives of $f$, denoted,
respectively, by $d_{\epsilon}^{-}f$ and $d_{\epsilon}^{+}f$, are defined by
\begin{equation}
d_{\epsilon}^{-}f(t)=\frac{f(t)-f(t-\epsilon)}{\epsilon}
\quad \text{ and } \quad
d_{\epsilon}^{+}f(t)=\frac{f(t+\epsilon)-f(t)}{\epsilon} \, .
\end{equation}
\end{definition}

\begin{remark}
The $\epsilon$-left and $\epsilon$-right quantum derivatives of a continuous
function $f$ correspond to the classical derivative of the $\epsilon$-mean
function $f_{\epsilon}^{\sigma}$ defined by
\begin{equation}
f_{\epsilon}^{\sigma}(t)=\frac{\sigma}{\epsilon}
\int_{t}^{t+\sigma\epsilon}f(s)ds\, ,
\quad \sigma=\pm \, .
\end{equation}
\end{remark}

The next operator generalizes the classical derivative.

\begin{definition}[The $\epsilon$-scale derivative \cite{cresson_greff}]
\label{def:scaleder}
Let $f\in \mathcal{C}^{0}\left(I,\mathbb{R}^{d}\right)$. For all $\epsilon>0$,
the $\epsilon$-scale derivative of $f$, denoted by $\frac{\Box_{\epsilon}f}{\Box t}$,
is defined by
\begin{gather}
\frac{\Box_{\epsilon}f}{\Box t}
=\frac{1}{2}\left[\left(d_{\epsilon}^{+}f
+d_{\epsilon}^{-}f\right)
+i\mu\left(d_{\epsilon}^{+}f
-d_{\epsilon}^{-}f\right)\right],
\end{gather}
where $i$ is the imaginary unit and $\mu \in \{-1,1,0,-i,i\}$.
\end{definition}

\begin{remark}
If $f$ is differentiable, then one can take the limit of the scale derivative
when $\epsilon$ goes to zero. We then obtain the classical derivative
$\frac{df}{dt}$ of $f$.
\end{remark}

We also need to extend the scale derivative to complex valued functions.

\begin{definition}[See \cite{cresson_greff}]
Let $f\in \mathcal{C}^{0}\left(I,\mathbb{C}^{d}\right)$ be a continuous
complex valued function. For all $\epsilon>0$, the $\epsilon$-scale derivative
of $f$, denoted by $\frac{\Box_{\epsilon}f}{\Box t}$, is defined by
\begin{gather}
\frac{{\Box}_{\epsilon}f}{{\Box}t}
=\frac{{\Box}_{\epsilon}\textrm{Re}(f)}{\Box t}
+i\frac{\Box_{\epsilon}\textrm{Im}(f)}{\Box t} \, ,
\end{gather}
where $\textrm{Re}(f)$ and $\textrm{Im}(f)$ denote the real
and imaginary part of $f$, respectively.
\end{definition}

In Definition~\ref{def:scaleder}, the $\epsilon$-scale derivative depends
on $\epsilon$, which is a free parameter related to the smoothing order
of the function. This brings many difficulties in applications to physics,
when one is interested in particular equations that do not depend
on an extra parameter. To solve these problems, the authors of \cite{cresson_greff}
introduced a procedure to extract information independent of $\epsilon$
but related with the mean behavior of the function.

\begin{definition}[See \cite{cresson_greff}]
Let ${\mathcal{C}^0_{conv}}\left(I\times ]0,1],\mathbb{R}^{d}\right)
\subseteq {\mathcal{C}^0}\left(I\times]0,1],\mathbb{R}^{d}\right)$
be such that for any function $f \in {\mathcal{C}^0_{conv}}\left(I\times ]0,1],
\mathbb{R}^{d}\right)$ the $\lim_{\epsilon\to 0}f(t,\epsilon)$ exists for any
$t\in I$. We denote by $E$ a complementary space of
${\mathcal{C}^0_{conv}}\left(I\times ]0,1],\mathbb{R}^{d}\right)$
in ${\mathcal{C}^0}\left(I\times ]0,1],\mathbb{R}^{d}\right)$.
We define the projection map $\pi$ by
\begin{equation}
\fonction{\pi}{{\mathcal{C}^0_{conv}}\left(I\times ]0,1],\mathbb{R}^{d}\right)
\oplus E}{\mathcal{C}^0_{conv}(I\times\left]0,1],
\mathbb{R}^{d}\right)}{f_{conv}+f_E}{f_{conv}}
\end{equation}
and the operator $\left< \cdot \right>$ by
\begin{equation}
\fonction{\left< \cdot \right>}{{\mathcal{C}^0}\left(I\times ]0,1],
\mathbb{R}^{d}\right)}{{\mathcal{C}^0}\left(I,
\mathbb{R}^{d}\right)}{f}{\left< f \right>: t\mapsto
\displaystyle\lim_{\epsilon\to 0}\pi(f)(t,\epsilon).}
\end{equation}
\end{definition}

The quantum derivative of $f$ without the dependence of $\epsilon$
is introduced in \cite{cresson_greff}.

\begin{definition}[See \cite{cresson_greff}]
The quantum derivative of $f$ in the space $\mathcal{C}^{0}\left(I,
\mathbb{R}^{d}\right)$ is given by
\begin{equation}
\label{eq:quanDer}
\frac{\Box f}{\Box t}=\left<\frac{{\Box_{\epsilon}}f}{\Box t} \right>.
\end{equation}
\end{definition}

The quantum derivative \eqref{eq:quanDer} has some nice properties. 
Namely, it satisfies a Leibniz rule and a version 
of the fundamental theorem of calculus.

\begin{theorem}[The quantum Leibniz rule \cite{cresson_greff}]
\label{theo:mult}
Let $\alpha+\beta>1$. For $f\in H^\alpha\left(I,\mathbb{R}^{d}\right)$
and $g\in H^\beta\left(I,\mathbb{R}^{d}\right)$, one has
\begin{equation}
\label{eq:leibniz}
\frac{\Box}{\Box t}(f\cdot g)(t)=\frac{\Box f(t)}{\Box t}
\cdot g(t)+f(t)\cdot\frac{\Box g(t)}{\Box t}\,.
\end{equation}
\end{theorem}

\begin{remark}
For $f\in \mathcal{C}^1\left(I, \mathbb{R}^{d}\right)$
and $g\in \mathcal{C}^1\left(I, \mathbb{R}^{d}\right)$,
one obtains from \eqref{eq:leibniz} the classical Leibniz rule:
$(f\cdot g)'=f'\cdot g+f\cdot g'$.
\end{remark}

\begin{definition}
We denote by $\mathcal{C}^1_\Box$ the set of continuous functions
$q\in \mathcal{C}^0([a,b],\mathbb{R}^d)$ such that $\frac{\Box q}{\Box t}
\in \mathcal{C}^0\left(I, \mathbb{R}^{d}\right)$.
\end{definition}

\begin{theorem}[The quantum version
of the fundamental theorem of calculus \cite{cresson_greff}]
Let $f\in \mathcal{C}^1_\Box([a,b],\mathbb{R}^d)$ be such that
\begin{equation}
\label{eq:divcond}
\lim_{\epsilon\to0}
\int_{a}^{b}\left(\frac{\Box_\epsilon f}{\Box t}\right)_E(t)dt=0.
\end{equation}
Then,
\begin{equation}
\int^{b}_{a} \frac{\Box f}{\Box t}(t)\, dt=f(b)-f(a).
\end{equation}
\end{theorem}

% ------------------------------------------

\subsection{Nondifferentiable calculus of variations}

In \cite{cresson_greff} the calculus of variations with quantum derivatives is
introduced and respective Euler--Lagrange equations derived
without the dependence of $\epsilon$.

\begin{definition}
An admissible Lagrangian $L$ is a continuous function
$\fonctionsansdef{L}{\R\times \R^d\times \C^d}{\C}$ such that $L(t,x,v)$
is holomorphic with respect to $v$ and differentiable with respect to $x$.
Moreover, $L(t,x,v) \in \mathbb{R}$ when $v\in\R^d$;
$L(t,x,v) \in \mathbb{C}$ when $v\in\C^d$.
\end{definition}

An admissible Lagrangian function
$\fonctionsansdef{L}{\R\times \R^d\times \C^d}{\C}$
defines a functional on $\mathcal{C}^1(I,\R^d)$, denoted by
\begin{equation}
\fonction{\mathcal{L}}{\mathcal{C}^1(I,\R^d)}{\R}{q}{\int_{a}^{b}L(t,q(t),\dot{q}(t))dt.}
\end{equation}
Extremals of the functional $\mathcal{L}$ can be characterized by the
well-known Euler--Lagrange equation (see, e.g., \cite{arno}).

\begin{theorem}
The extremals $q\in \mathcal{C}^1(I,\R^d)$ of $\mathcal{L}$ coincide with
the solutions of the Euler--Lagrange equation
\begin{equation}
\frac{d}{dt}\left[\frac{\partial L}{\partial v}\left(t,q(t),\dot{q}(t)\right)\right]
=\frac{\partial L}{\partial x}\left(t,q(t),\dot{q}(t)\right).
\end{equation}
\end{theorem}

The nondifferentiable embedding procedure allows us to define a natural
extension of the classical Euler--Lagrange equation in the nondifferentiable context.

\begin{definition}[See \cite{cresson_greff}]
The nondifferentiable Lagrangian functional $\mathcal{L}_\Box$
associated with $\mathcal{L}$ is given by
\begin{equation}
\label{eq:NDL}
\fonction{\mathcal{L}_\Box}{\mathcal{C}^1_\Box(I,\R^d)}{\R}{q}{\int_{a}^{b}
L\left(s,q(s),\frac{\Box q(s)}{\Box t}\right)ds.}
\end{equation}
\end{definition}

Let $H^\beta_0:=\{h\in H^\beta(I,\R^d), h(a)=h(b)=0\}$ and
$q\in H^\alpha(I,\R^d)$ with $\alpha+\beta > 1$. A $H^\beta_0$-variation
of $q$ is a function of the form $q+h$, where $h\in H^\beta_O$.
We denote by $D\mathcal{L}_\Box(q)(h)$ the quantity
\begin{equation}
\lim_{\epsilon \to0 }
\frac{\mathcal{L}_\Box(q+\epsilon h)-\mathcal{L}_\Box(q)}{\epsilon}
\end{equation}
if there exists the so-called Fr\'echet derivative of $\mathcal{L}_\Box$
at point $q$ in direction $h$.

\begin{definition}[Nondifferentiable extremals]
A $H^\beta_0$-extremal curve of the functional $\mathcal{L}_\Box$ is a curve
$q \in H^\alpha(I,\R^d)$ satisfying $D\mathcal{L}_\Box(q)(h)=0$
for any $h\in H^\beta _0$.
\end{definition}

\begin{theorem}[Nondifferentiable Euler--Lagrange equations \cite{cresson_greff}]
Let $0<\alpha,\, \beta<1$ with $\alpha+\beta>1$. Let $L$ be an admissible Lagrangian
of class $\mathcal{C}^2$. We assume that
$\gamma\in H^{\alpha}\left(I,\mathbb{R}^{d}\right)$, such that
$\frac{\Box \gamma}{\Box t}\in H^{\alpha}\left(I, \mathbb{R}^{d}\right)$. Moreover,
we assume that $L\left(t,\gamma(t),\frac{\Box \gamma(t)}{\Box t}\right)h(t)$
satisfies condition \eqref{eq:divcond} for all
$h\in H^{\beta}_0\left(I, \mathbb{R}^{d}\right)$. A curve $\gamma$ satisfying
the nondifferentiable Euler--Lagrange equation
\begin{equation}
\frac{\Box }{\Box t}\left[\frac{\partial L}{\partial v}\left(t,\gamma(t),
\frac{\Box \gamma(t)}{\Box t}\right)\right]
= \frac{\partial L}{\partial x}
\left(t,\gamma(t),\frac{\Box \gamma(t)}{\Box t}\right)
\end{equation}
is an extremal curve of the functional \eqref{eq:NDL}.
\end{theorem}

% ------------------------------------------

\section{Reminder about Hamiltonian systems}
\label{sec:reminderhamsys}

We now recall the main concepts and results of both classical
and Cresson's nondifferentiable Hamiltonian systems.

% ------------------------------------------

\subsection{Classical Hamiltonian systems}

Let $L$ be an admissible Lagrangian function. If $L$ satisfies
the so-called Legendre property, then we can associate to $L$
a Hamiltonian function denoted by $H$.

\begin{definition}
Let $L$ be an admissible Lagrangian function. The Lagrangian $L$ is said
to satisfy the Legendre property if the mapping
$v\mapsto \frac{\partial L}{\partial v}(t,x,v)$ is invertible for any
$(t, q, v)\in I\times \R^d \times \C^d$.
\end{definition}

If we introduce a new variable
\begin{equation}
p=\frac{\partial L}{\partial v}(t,q,v)
\end{equation}
and $L$ satisfies the Legendre property,
then we can find a function $f$ such that
\begin{equation}
v=f(t,q,p).
\end{equation}
Using this notation, we have the following definition.

\begin{definition}
Let $L$ be an admissible Lagrangian function satisfying the Legendre property.
The Hamiltonian function $H$ associated with $L$ is given by
\begin{equation}
\label{eq:hamfun}
\fonction{H}{\R\times\R^d\times\C^d}{\C}{(t,q,p)}{H(t,q,p)
=pf(t,q,p)-L(t,q,f(t,q,p)).}
\end{equation}
\end{definition}

We have the following theorem (see, e.g., \cite{arno}).

\begin{theorem}[Hamilton's least-action principle]
The curve $(q,p)\in\mathcal{C}(I,\R^d)\times \mathcal{C}(I,\C^d)$
is an extremal of the Hamiltonian functional
\begin{equation}
\mathcal{H}(q,p)=\int_{a}^{b}p(t)\dot{q}(t)-H(t,q(t),p(t))dt
\end{equation}
if and only if it satisfies the Hamiltonian system associated with $H$ given by
\begin{equation}
\left\{
\begin{aligned}
\dot{q}(t)&=\frac{\partial H(t,q(t),p(t))}{\partial p} \\
\dot{p}(t)&=-\frac{\partial H(t,q(t),p(t))}{\partial q}
\end{aligned}
\right.
\label{def_hamilton}
\end{equation}
called the Hamiltonian equations.
\end{theorem}

A vectorial notation is obtained for the Hamiltonian equations in posing
$z=(q,p)^\mathsf{T}$ and $\nabla H = (\frac{\partial H}{\partial q},
\frac{\partial H}{\partial p} )^\mathsf{T}$, where $\mathsf{T}$ denotes
the transposition. The Hamiltonian equations are then written as
\begin{equation}
\frac{dz(t)}{dt} = J \cdot \nabla H(t,z(t)),
\end{equation}
where 
\begin{equation}
\label{eq:J}
J = \begin{pmatrix} 0 & I_d \\ -I_d & 0 \end{pmatrix}
\end{equation}
denotes the symplectic matrix with $I_d$ being the identity matrix on $\R^d$.

% ------------------------------------------

\subsection{Nondifferentiable Hamiltonian systems}

The nondifferentiable embedding induces a change in the phase space
with respect to the classical case. As a consequence, we have to work
with variables $(x, p)$ that belong to $\R^d\times \C^d$ and not only
to $\R^d\times \R^d$, as usual.

\begin{definition}[Nondifferentiable embedding
of Hamiltonian systems \cite{cresson_greff}]
The nondifferentiable embedded Hamiltonian
system \eqref{def_hamilton} is given by
\begin{equation}
\left\{
\begin{aligned}
\frac{\Box q(t)}{\Box t}&=\frac{\partial H(t,q(t),p(t))}{\partial p} \\
\frac{\Box p(t)}{\Box t}&=-\frac{\partial H(t,q(t),p(t))}{\partial q}
\end{aligned}
\right.
\label{def_NDhamilton}
\end{equation}
and the embedded Hamiltonian functional $\mathcal{H}_\Box$
is defined on $H^\alpha(I,\R^d)\times H^\alpha(I,\C^d)$ by
\begin{equation}
\label{eq:emb_H:functional}
\mathcal{H}_\Box(q,p)
=\int_{a}^{b}\left(p(t)\frac{\Box q(t)}{\Box t}-H(t,q(t),p(t))\right) dt.
\end{equation}
\end{definition}
The nondifferentiable calculus of variations allows us to derive
the extremals for $\mathcal{H}_\Box$.

\begin{theorem}[Nondifferentiable
Hamilton's least-action principle \cite{cresson_greff}]
Let $0<\alpha,\, \beta<1$ with $\alpha+\beta>1$.
Let $L$ be an admissible $\mathcal{C}^2$-Lagrangian. We assume that
$\gamma\in H^{\alpha}\left(I,\mathbb{R}^{d}\right)$, such that
$\frac{\Box \gamma}{\Box t}\in H^{\alpha}\left(I, \mathbb{R}^{d}\right)$.
Moreover, we assume that $L\left(t,\gamma(t),\frac{\Box \gamma(t)}{\Box t}\right)h(t)$
satisfies condition \eqref{eq:divcond} for all
$h\in H^{\beta}_0\left(I, \mathbb{R}^{d}\right)$. Let $H$ be the corresponding
Hamiltonian defined by \eqref{eq:hamfun}. A curve $\gamma\mapsto(t,q(t),p(t))
\in I\times \R^d\times\C^d$ solution of the nondifferentiable Hamiltonian
system \eqref{def_NDhamilton} is an extremal of the functional
\eqref{eq:emb_H:functional} over the space of variations
$V=H^\beta_0(I,\R^d)\times H^\beta_0(I,\C^d)$.
\end{theorem}

% ------------------------------------------

\section{Nondifferentiable Helmholtz problem}
\label{sec:nondiffhelmholtz}

In this section, we solve the inverse problem of the nondifferentiable
calculus of variations in the Hamiltonian case. We first recall the usual way
to derive the Helmholtz conditions following the presentation made
by Santilli \cite{santilli}. Two main derivations are available:
\begin{itemize}
\item[(i)] the first is related to the characterization of Hamiltonian systems
via the \emph{symplectic two-differential form} and the fact that
by duality the associated one-differential form to a Hamiltonian vector field
is closed---the so called \emph{integrability conditions};

\item[(ii)] the second uses the characterization of Hamiltonian systems
via the self-adjointness of the Fr\'echet derivative of the differential
operator associated with the equation---the so called \emph{Helmholtz conditions}.
\end{itemize}
Of course, we have coincidence of the two procedures in the classical case.
As there is no analogous of differential form in the framework of Cresson's
quantum calculus, we follow the second way to obtain the nondifferentiable
analogue of the Helmholtz conditions. For simplicity, we consider
a time-independent Hamiltonian. The time-dependent case can be done in the same way.

% ------------------------------------------

\subsection{Hemlholtz conditions for classical Hamiltonian systems}

In this section we work on $\R^{2d}$, $d \ge 1$, $d \in \N$.

\subsubsection{Symplectic scalar product}

The \emph{symplectic scalar product} $\langle \cdot,\cdot \rangle_J$ is defined by
\begin{equation}
\langle X,Y\rangle_J = \langle X,J\cdot Y\rangle
\end{equation}
for all $X,Y \in \R^{2d}$, where $\langle \cdot,\cdot \rangle$ denotes
the usual scalar product and $J$ is the symplectic matrix \eqref{eq:J}. 
We also consider the $L^2$ symplectic scalar product
induced by $\langle \cdot,\cdot \rangle_J$ defined for
$f,g \in \mathcal{C}^0([a,b],\R^{2d})$ by
\begin{equation}
\langle f,g \rangle_{L^2,J}=\displaystyle\int_{a}^{b} \langle f(t),g(t) \rangle_Jdt \ .
\end{equation}

% ------------------------------------------

\subsubsection{Adjoint of a differential operator}

In the following, we consider first-order differential equations of the form
\begin{equation}
\label{eq:diffsys}
\frac{d}{dt}\begin{pmatrix}q \\
p
\end{pmatrix}
= \begin{pmatrix}
X_q(q,p) \\
X_p(q,p)
\end{pmatrix},
\end{equation}
where the vector fields $X_q$ and $X_p$ 
are $\mathcal{C}^1$ with respect to $q$ and $p$.
The associated differential operator is written as
\begin{equation}
O_X(q,p) =
\begin{pmatrix}
\dot{q} - X_q(q,p) \\
\dot{p} - X_p(q,p)
\end{pmatrix}.
\label{eq:operatorO}
\end{equation}
A \emph{natural} notion of adjoint for a differential operator
is then defined as follows.

\begin{definition}
Let $\fonctionsansdef{A}{\mathcal{C}^1([a,b],
\R^{2d})}{\mathcal{C}^1([a,b],\R^{2d})}$.
We define the adjoint $A^*_J$ of $A$ with respect
to $\langle\cdot,\cdot\rangle_{L^2,J}$ by
\begin{equation}
\langle A \cdot f, g\rangle_{L^2,J}
= \langle A^{*}_J \cdot g, f\rangle_{L^2,J} \ .
\end{equation}
\end{definition}

An operator $A$ will be called \emph{self-adjoint} if $A=A^{*}_J$
with respect to the $L^2$ symplectic scalar product.

% ------------------------------------------

\subsubsection{Hamiltonian Helmholtz conditions}

The Helmholtz conditions in the Hamiltonian case
are given by the following result
(see Theorem~3.12.1, p.~176--177 in \cite{santilli}).

\begin{theorem}[Hamiltonian Helmholtz theorem]
Let $X(q,p)$ be a vector field defined by $X(q,p)^\mathsf{T}= (X_q(q,p) , X_p(q,p) )$.
The differential equation \eqref{eq:diffsys} is Hamiltonian
if and only if the associated differential operator $O_X$ given
by \eqref{eq:operatorO} has a  self-adjoint Fr\'echet derivative with respect
to the $L^2$ symplectic scalar product. In this case the Hamiltonian is given by
\begin{equation}
H(q,p)=\int_{0}^{1}\left[p \cdot X_q(\lambda q, \lambda p)
- q\cdot X_p(\lambda q, \lambda p) \right]d\lambda.
\end{equation}
\end{theorem}

The conditions for the self-adjointness of the differential operator can be
made \emph{explicit}. They coincide with the \emph{integrability conditions}
characterizing the exactness of the one-form associated with the vector field
by duality (see \cite{santilli}, Theorem~2.7.3, p.~88).

\begin{theorem}[Integrability conditions]
Let $X(q,p)^\mathsf{T}= (X_q(q,p) , X_p(q,p) )$ be a vector field.
The differential operator $O_X$ given by \eqref{eq:operatorO} has
a self-adjoint Fr\'echet derivative with respect to the $L^2$
symplectic scalar product if and only if
\begin{equation}
\frac{\partial X_q}{\partial q} + \left(\frac{\partial X_p}{\partial p}
\right)^\mathsf{T} = 0, \quad \frac{\partial X_q}{\partial p}
\ \text{and} \ \frac{\partial X_p}{\partial q} \ \text{are symmetric}.
\end{equation}
\end{theorem}

% ------------------------------------------

\subsection{Hemlholtz conditions for nondifferentiable Hamiltonian systems}

The previous scalar products extend naturally to complex valued functions.
Let $0<\alpha<1$ and let $(q,p)\in H^{\alpha}\left(I,\mathbb{R}^{d}\right)
\times H^{\alpha}\left(I,\mathbb{C}^{d}\right)$, such that $\frac{\Box q}{\Box t}
\in H^{\alpha}\left(I, \mathbb{C}^{d}\right)$ and $\frac{\Box p}{\Box t}
\in H^{\alpha}\left(I, \mathbb{C}^{d}\right)$. We consider first-order
nondifferential equations of the form
\begin{equation}
\label{eq:NDdiffsys}
\frac{\Box}{\Box t} \left(
\begin{aligned}
q \\
p
\end{aligned}
\right)=\left(
\begin{aligned}
X_q(q,p) \\ X_p(q,p)
\end{aligned}
\right).
\end{equation}
The associated quantum differential operator is written as
\begin{equation}
O_{\Box,X}(q,p) = \left(
\begin{aligned}
\frac{\Box q}{\Box t} - X_q(q,p) \\ \frac{\Box p}{\Box t} - X_p(q,p)
\end{aligned}
\right).
\label{eq:NDoperatorO}
\end{equation}
A \emph{natural} notion of adjoint for a quantum differential operator
is then defined.

\begin{definition}
Let $\fonctionsansdef{A}{\mathcal{C}^1_\Box([a,b],
\C^{2d})}{\mathcal{C}^1_\Box([a,b],\C^{2d})}$. We define the
adjoint $A^*_J$ of $A$ with respect to $\langle\cdot,\cdot\rangle_{L^2,J}$ by
\begin{equation}
\langle A \cdot f, g\rangle_{L^2,J}
= \langle A^{*}_J \cdot g, f\rangle_{L^2,J} \ .
\end{equation}
\end{definition}

An operator $A$ will be called \emph{self-adjoint} if $A=A^{*}_J$
with respect to the $L^2$ symplectic scalar product. We can now
obtain the adjoint operator associated with $O_{\Box,X}$.

\begin{proposition}
Let $\beta$ be such that $\alpha+\beta>1$.
Let $(u,v)\in H^{\beta}_0\left(I,\mathbb{R}^{d}\right)
\times H^{\beta}_0\left(I,\mathbb{C}^{d}\right)$, such that
$\frac{\Box u}{\Box t}\in H^{\alpha}\left(I, \mathbb{C}^{d}\right)$
and $\frac{\Box v}{\Box t}\in H^{\alpha}\left(I, \mathbb{C}^{d}\right)$.
The Fr\'echet derivative $DO_{\Box,X}$ of \eqref{eq:NDoperatorO}
at $(q,p)$ along $(u,v)$ is then given by
\begin{equation}
DO_{\Box,X}(q,p)(u,v)
= \left(
\begin{aligned}
&\frac{\Box u}{\Box t}
-\frac{\partial X_q }{\partial q} \cdot u
-\frac{\partial X_q}{\partial p}\cdot v \\
&\frac{\Box v}{\Box t} -\frac{\partial X_p}{\partial q}
\cdot u -\frac{\partial X_p}{\partial p}\cdot v
\end{aligned}\right).
\end{equation}
Assume that $u\cdot h$ and $v\cdot h$ satisfy condition \eqref{eq:divcond}
for any $h\in H^\beta_0(I,\C^d)$. In consequence, the adjoint
$DO^*_{\Box,X}$ of $DO_{\Box,X}(q,p)$ with respect to the
$L^2$ symplectic scalar product is given by
\begin{equation}
DO^*_{\Box,X}(q,p)(u,v)
= \left(
\begin{aligned}
\frac{\Box u}{\Box t} +\left(\frac{\partial X_p }{\partial p}\right)^\mathsf{T}
\cdot u -\left(\frac{\partial X_q }{\partial p}\right)^\mathsf{T} \cdot v \\
\frac{\Box v}{\Box t} -\left(\frac{\partial X_p }{\partial q}\right)^\mathsf{T}
\cdot u +\left(\frac{\partial X_q }{\partial q}\right)^\mathsf{T} \cdot v
\end{aligned}
\right).
\end{equation}
\end{proposition}

\begin{proof}
The expression for the Fr\'echet derivative of \eqref{eq:NDoperatorO} at $(q,p)$
along $(u,v)$ is a simple computation. Let $(w,x)\in H^{\beta}_0\left(I,
\mathbb{R}^{d}\right)\times H^{\beta}_0\left(I,\mathbb{C}^{d}\right)$ be
such that $\frac{\Box w}{\Box t}\in H^{\alpha}\left(I, \mathbb{C}^{d}\right)$
and $\frac{\Box x}{\Box t}\in H^{\alpha}\left(I, \mathbb{C}^{d}\right)$.
By definition, we have
\begin{multline}
\langle DO_{\Box,X}(q,p)(u,v), (w,x)\rangle_{L^2,J}
=\int_{a}^{b} \left[\frac{\Box u}{\Box t}\cdot x
-\left(\frac{\partial X_q}{\partial q} \cdot u\right)\cdot x
-\left(\frac{\partial X_q }{\partial p}\cdot v\right)\cdot x\right. \\
\left. -\frac{\Box v}{\Box t}\cdot w +\left(\frac{\partial X_p}{\partial q}
\cdot u\right)\cdot w 
+\left(\frac{\partial X_p}{\partial p}\cdot v\right)\cdot w\right] dt.
\end{multline}
As $u\cdot h$ and $v\cdot h$ satisfy condition \eqref{eq:divcond} for any
$h\in H^\beta_0(I,\C^d)$, using the quantum Leibniz rule and the quantum version
of the fundamental theorem of calculus, we obtain
\begin{equation}
\begin{aligned}
&\int_{a}^{b}\frac{\Box u}{\Box t}\cdot b \ dt
= \int_{a}^{b}-u\cdot \frac{\Box b}{\Box t} \ dt, \\
&\int_{a}^{b}\frac{\Box v}{\Box t}\cdot a \ dt
= \int_{a}^{b}-u\cdot \frac{\Box b}{\Box t} \ dt.
\end{aligned}
\end{equation}
Then,
\begin{multline}
\langle DO_{\Box,X}(q,p)(u,v), (w,x)\rangle_{L^2,J}
=\int_{a}^{b} \left[-u\cdot \left(\frac{\Box x}{\Box t}
-\left(\frac{\partial X_p}{\partial q}\right)^\mathsf{T}\cdot w
+ \left(\frac{\partial X_q}{\partial q}\right)^\mathsf{T}\cdot x\right)\right.\\
\left. +v\cdot \left(\frac{\Box w}{\Box t}
+\left(\frac{\partial X_p}{\partial p}\right)^\mathsf{T}\cdot w
- \left(\frac{\partial X_q}{\partial p}\right)^\mathsf{T}\cdot x\right)\right] dt.
\end{multline}
By definition, we obtain the expression of the adjoint $DO^*_{\Box,X}$ of
$DO_{\Box,X}(q,p)$ with respect to the $L^2$ symplectic scalar product.
\end{proof}

In consequence, from a direct identification, we obtain the nondifferentiable
self-adjointess conditions called \emph{Helmholtz's conditions}. As in the
classical case, we call these conditions
\emph{nondifferentiable integrability conditions}.

\begin{theorem}[Nondifferentiable integrability conditions]
Let $X(q,p)^\mathsf{T}= (X_q(q,p) , X_p(q,p) )$ be a vector field. The
differential operator $O_{\Box,X}$ given by \eqref{eq:NDoperatorO} has
a self-adjoint Fr\'echet derivative with respect to the symplectic
scalar product if and only if
\begin{align}
&\frac{\partial X_q}{\partial q} + \left(\frac{\partial X_p}{\partial p}
\right)^\mathsf{T} = 0, \label{HC1} \tag{HC1} \\
&\frac{\partial X_q}{\partial p}
\ \text{and} \ \frac{\partial X_p}{\partial q}
\ \text{are symmetric} \label{HC2} \tag{HC2}.
\end{align}
\end{theorem}

\begin{remark}
One can see that the Helmholtz conditions are the same as in the classical,
discrete, time-scale, and stochastic cases. We expected such a result because
Cresson's quantum calculus provides a quantum Leibniz rule and a quantum version
of the fundamental theorem of calculus. If such properties of an underlying
calculus exist, then the Helmholtz conditions will always be the same
up to some conditions on the working space of functions.
\end{remark}

We now obtain the main result of this paper, which is the
\emph{Helmholtz theorem for nondifferentiable Hamiltonian systems}.

\begin{theorem}[Nondifferentiable Hamiltonian Helmholtz theorem]
Let $X(q,p)$ be a vector field defined by $X(q,p)^\mathsf{T}= (X_q(q,p) , X_p(q,p) )$.
The nondifferentialble system of equations \eqref{eq:NDdiffsys} is Hamiltonian
if and only if the associated quantum differential operator $O_{\Box,X}$ given
by \eqref{eq:NDoperatorO} has a  self-adjoint Fr\'echet derivative with respect
to the $L^2$ symplectic scalar product. In this case, the Hamiltonian is given by
\begin{equation}
\label{ham_form}
H(q,p)=\int_{0}^{1}\left[p \cdot X_q(\lambda q, \lambda p)
- q\cdot X_p(\lambda q, \lambda p) \right]d\lambda.
\end{equation}
\end{theorem}

\begin{proof}
If $X$ is Hamiltonian, then there exists a function $\fonctionsansdef{H}{\R^d
\times\C^d}{\C}$ such that $H(q,p)$ is holomorphic with respect to $v$
and differentiable with respect to $q$ and $X_q=\frac{\partial H}{\partial p}$
and $X_p=-\frac{\partial H}{\partial q}$. The nondifferentiable integrability
conditions are clearly verified using Schwarz's lemma.
Reciprocally, we assume that $X$ satisfies the nondifferentiable integrability
conditions. We will show that $X$ is Hamiltonian with respect to the Hamiltonian
\begin{equation}
H(q,p)=\int_{0}^{1}\left[p \cdot X_q(\lambda q, \lambda p)
- q\cdot X_p(\lambda q, \lambda p) \right]d\lambda,
\end{equation}
that is, we must show that
\begin{equation}
X_q(q,p)=\frac{\partial H(q,p)}{\partial p}
\quad \text{and} \quad X_p(q,p)=-\frac{\partial H(q,p)}{\partial q}.
\end{equation}
We have
\begin{equation}
\begin{aligned}
\frac{\partial H(q,p)}{\partial q} &=\int_{0}^{1} \left[p
\cdot \lambda \frac{\partial X_q(\lambda q,\lambda p)}{\partial q}
- X_p(\lambda q,\lambda p) - q\cdot \lambda \frac{\partial X_p(\lambda q,
\lambda p)}{\partial q} \right]d\lambda, \\
\frac{\partial H(q,p)}{\partial p} &=\int_{0}^{1} \left[X_p(\lambda q,\lambda p)
+ p\cdot \lambda \frac{\partial X_q(\lambda q,\lambda p)}{\partial p}
- q\cdot \lambda \frac{\partial X_p(\lambda q,\lambda p)}{\partial p}
\right]d\lambda.
\end{aligned}
\end{equation}
Using the nondifferentiable integrability conditions, we obtain
\begin{equation}
\begin{aligned}
\frac{\partial H(q,p)}{\partial q} &=\int_{0}^{1}
- \frac{\partial \left(\lambda X_p(\lambda q,
\lambda p)\right)}{\partial \lambda} d\lambda = - X_p(q,p), \\
\frac{\partial H(q,p)}{\partial q} &=\int_{0}^{1}
\frac{\partial \left(\lambda X_q(\lambda q,\lambda p)\right)}{\partial
\lambda} d\lambda = X_q(q,p),
\end{aligned}
\end{equation}
which concludes the proof.
\end{proof}

% ------------------------------------------

\section{Applications}
\label{sec:Appl}

We now provide two illustrative examples of our results: one with the formulation
of dynamical systems with linear parts and another with Newton's equation,
which is particularly useful to study partial differentiable equations such
as the Navier--Stokes equation. Indeed, the Navier--Stokes equation
can be recovered from a Lagrangian structure with Cresson's quantum calculus
\cite{cresson_greff}. For more applications see \cite{cresson-pierret2}.

Let $0<\alpha<1$ and let $(q,p)\in H^{\alpha}\left(I,\mathbb{R}^{d}\right)
\times H^{\alpha}\left(I,\mathbb{C}^{d}\right)$ be such that
$\frac{\Box q}{\Box t}\in H^{\alpha}\left(I, \mathbb{C}^{d}\right)$
and $\frac{\Box q}{\Box t}\in H^{\alpha}\left(I, \mathbb{C}^{d}\right)$.

% ------------------------------------------

\subsection{The linear case}

Let us consider the discrete nondifferentiable system
\begin{gather}
\left\{
\begin{split}
\frac{\Box q}{\Box t}&=\alpha q + \beta p,\\
\frac{\Box p}{\Box t}&=\gamma q + \delta p,
\end{split}
\right.
\label{ex1_eq}
\end{gather}
where $\alpha$, $\beta$, $\gamma$ and $\delta$ are constants. The Helmholtz
condition \eqref{HC2} is clearly satisfied. However, the system \eqref{ex1_eq}
satisfies the condition \eqref{HC1} if and only if $\alpha+\delta=0$.
As a consequence, linear Hamiltonian nondifferentiable equations are of the form
\begin{gather}
\left\{
\begin{split}
\frac{\Box q}{\Box t}&=\alpha q + \beta p, \\
\frac{\Box p}{\Box t}&=\gamma q -\alpha p.
\end{split}
\right.
\end{gather}
Using formula \eqref{ham_form}, we compute explicitly
the Hamiltonian, which is given by
\begin{equation}
H(q,p)=\frac{1}{2}\left(\beta p^2-\gamma q^2\right) + \alpha q\cdot p \ .
\end{equation}

% ------------------------------------------

\subsection{Newton's equation}

Newton's equation (see \cite{arno}) is given by
\begin{gather}
\left\{
\begin{split}
\dot{q}&=p/m, \\
\dot{p}&=-U'(q),
\end{split}
\right.
\label{ex2_eqcont}
\end{gather}
with $m\in \R^+$ and $q,p\in \R^d$. This equation possesses
a natural Hamiltonian structure with the Hamiltonian given by
\begin{equation}
\label{eq:H:ex:New}
H(q,p)=\frac{1}{2m}p^2+U(q).
\end{equation}
Using Cresson's quantum calculus, we obtain
a natural nondifferentiable system given by
\begin{gather}
\left\{
\begin{split}
\frac{\Box q}{\Box t}&=p/m, \\
\frac{\Box p}{\Box t}&=-U'(q).
\end{split}
\right.
\label{ex2_eq1}
\end{gather}
The Hamiltonian Helmholtz conditions are clearly satisfied.

\begin{remark}
It must be noted that Hamiltonian \eqref{eq:H:ex:New}
associated with \eqref{ex2_eq1}
is recovered by formula \eqref{ham_form}.
\end{remark}

% ------------------------------------------

\section{Conclusion}
\label{sec:conclu}

We proved a Helmholtz theorem for nondifferentiable equations,
which gives necessary and sufficient conditions for the existence
of a Hamiltonian structure. In the affirmative case,
the Hamiltonian is given. Our result extends
the results of the classical case when restricting attention
to differentiable functions. An important complementary result
for the nondifferentiable case is to obtain the
Helmholtz theorem in the Lagrangian case.
This is nontrivial and will be subject of future research.

% ------------------------------------------

\section*{Acknowledgments}

This work was supported by FCT and CIDMA
through project UID/MAT/04106/2013. The first author
is grateful to CIDMA and DMat-UA for the hospitality
and good working conditions during his visit
at University of Aveiro. The authors would like 
to thank an anonymous referee for careful reading 
of the submitted paper and for useful suggestions.

% ------------------------------------------

% ------------------------------------------

\end{document}